\documentclass[11pt]{article}
\usepackage{geometry}
\usepackage{latexsym,amsmath,amssymb,amsfonts,enumerate,eurosym}
\usepackage{color}
\usepackage{multicol}
\usepackage{hyperref}
\usepackage{wrapfig}
\usepackage{colonequals}
\usepackage{epsfig}
\usepackage[T1]{fontenc}
\usepackage{amsthm}
\usepackage{graphics}
\usepackage{graphicx}
\usepackage{tikz}
\usetikzlibrary{matrix,decorations.pathreplacing,calc, positioning,fit}
\usepackage[table]{colortbl}
\usepackage{pgfplots}
\usepackage{subfig}
\usepackage{stmaryrd}
\theoremstyle{remark}

\theoremstyle{plain}
\newtheorem{theorem}{Theorem}[section]
\newtheorem{lemma}[theorem]{Lemma}
\newtheorem{proposition}[theorem]{Proposition}
\newtheorem{defn}{Definition}[section]

\newtheorem{corollary}[theorem]{Corollary}
\newtheorem{conjecture}{Conjecture}[section]

\newtheorem{ex}{Example}[section]
\newtheorem{notation}{Notation}[section]

\newcommand{\PaperTitle}[1]{#1}
\newcommand{\BookTitle}[1]{\emph{#1}}
\newcommand{\JournalName}[1]{\emph{#1}}
\newcommand{\BookPublisher}[1]{#1}

\newcommand{\bailout}{%
  \ifnum\value{page}>1
  \typeout{Error! More than one page}
  \undefinedcommand
  \fi
  \ifdefined\MSCok\relax\else\errmessage{missing MSC}\fi
}

\def\rk{\mathrm{rk}}

\title{Matroidal Cycles and Hypergraph Families}
\author{Ragnar Freij-Hollanti and Patricija Šapokaitė}
\date{\today}
\begin{document}
\maketitle

\begin{abstract}
We propose a novel definition of hypergraphical matroids, defined for arbitrary hypergraphs, simultaneously generalizing previous definitions for regular hypergraphs (Main, 1978), and for the hypergraphs of circuits of a matroid (Freij-Hollanti, Jurrius, Kuznetsova, 2023). As a consequence, we obtain a new notion of cycles in hypergraphs, and hypertrees. We give an equivalence relation on hypergraphs, according to when their so-called matroidal closures agree. Finally, we characterize hypergraphs that are isomorphic to the circuit hypergraphs of the associated matroids.
\end{abstract}

\section{Introduction}
\label{sec1}

Hypergraphs, {\em i.e.} incidence structures between vertices and {\em hyperedges}, each of which can contain an arbitrary number of vertices, are a very natural generalization of graphs. However, the definition is so general that even the most obvious graph theory notions have a long list of possible generalizations, suitable for different use cases. To highlight this phenomenon, there is not a well agreed-upon definition of hypergraph cycles; the most prominent such definitions in the literature are Berge-cyclicity \cite{C_Berge:bib3} and $\alpha$-cyclicity \cite{C_Beeri:bib4}.

These definitions generalize some aspects of what it means to be a cycle in a graph, but fail to capture the aspect where cycles indicate {\em dependency} between edges. Loosely speaking, we want to view edges in hypergraphs as describing (partial) information about some relation between its nodes, and group such edges according to what information (or which edges) can be recovered. More precisely, we want a notion of hypergraph cyclicity such that the edge sets of cycles form the circuits of some matroid. Neither Berge-cycles nor $\alpha$-cycles have this desirable property.

Two previous attempts have been made to define {\em hypergraphical matroids}, by Main in 1978 \cite{R_Main:bib9} and by Lorea in 1975 \cite{M_Lorea:bib10}, both of which agree with traditional graphical matroids for $2$-regular hypergraphs, i.e. graphs. Main's matroid is defined for any {\em $k$-regular} hypergraph \cite{R_Main:bib9} and is representable over any field. Lorea's definition applies to arbitrary hypergraphs, but gives a matroid where hypertrees bases can have very complicated structure, and where circuits are few and far between.

It was pointed out by Santiago Guzmán-Pro~\cite{S_Guzman-Pro:bib8} that the special case of Main's the hypergraphical matroid of the complete $k$-regular hypergraph $K^k_n$ equals the combinatorial derived matroid of the uniform matroid $U_n^{k-1}$, as defined in~\cite{R_Freij-Hollanti:bib5}. Inspired by this rather simple observation, we propose a new definition of a hypergraphical matroid for an arbitrary hypergraph.

Our definition goes in two steps, where we first build, from a hypergraph $H=(V,\mathcal{E})$ a new hypergraph $(\mathcal{E},\mathcal{F}_0)$ on its edges, and then use this ``derived hypergraph'' as a ``skeleton'' for a matroid $\delta H = (\mathcal{E},\mathcal{F})$ on $\mathcal{E}$. In the case where $H$ is itself the circuit hypergraph of a matroid on $V$, our definition agrees with the combinatorial derived matroid defined in~\cite{R_Freij-Hollanti:bib5}. The operation $$(\mathcal{E},\mathcal{F}_0)\leadsto (\mathcal{E},\mathcal{F}),$$ which we think of as a closure operation on hypergraph (although it is not a closure operation in any categorical sense), is defined for arbitrary hypergraphs, and allows us to also see any hypergraph as a skeleton for a matroid on the same ground set.

The rest of the paper is structured as follows: In Section 2 we define our notion of matroidal cycles in hypergraphs. In Section 4, these cycles are the building blocks when we define hypergraphical matroids, using the closure operation that we introduce in Section 3, and which may have independent interest. In Section 5 and 6, we study the structures of cycle-free hypergraphs, categorize hypergraphs according to their closure, and take first steps towards characterizing which matroids can occur as hypergraphical matroids. Finally, Section 7 describes the sequences obtained by repeatedly deriving matroids. It is shown that this sequence diverges, except for matroidal trees and the special case of the uniform matroid $U_4^2$

\section{Matroidal cycles}

We begin this section by establishing a connection between matroids and hypergraphs, and in particular, interpreting the former as a special case of the latter, via the circuit set.

\begin{defn}
A hypergraph is a pair $(V,\mathcal{E})$, where $V$ is a set of {\em vertices} and $\mathcal{E}\subseteq \mathcal{P}(V)$ is a set of {\em hyperedges}. A hypergraph is {\em simple} if there are not two edges $e,f\in\mathcal{E}$ with $e\subseteq f$. A hypergraph is $k$-regular if $|e|=k$ for all $e\in\mathcal{E}$.
\end{defn}
We recover a simple graph as a $2$-regular simple hypergraph. The well known matroid axioms (via circuit sets) from \cite{O_James:bib11} can now be written in terms of hypergraphs:
\begin{defn}\label{MatroidDef}
    A simple hypergraph $M=(V,\mathcal{E})$ is a {\em matroid} if for every $C_1, C_2\in\mathcal{E}$ with $C_1\neq C_2$, and every $v \in C_1 \cap C_2$, there exists $C_3\in \mathcal{E}$ such that $C_3\subseteq C_1\cup C_2\setminus \{v\}$. The set $\mathcal{E}$ is the {\em circuit set} of $M$.
\end{defn}

It should be pointed out, that in the matroid literature, the ground set of a matroid is often denoted by $E$, and the circuit set by $\mathcal{C}$. Here, we choose the notation to be more indicative of the hypergraph interpretation.

There are multiple ways the cycles of hypergraphs have been defined in the literature. The most prominent of those are Berge-cyclicity \cite{C_Berge:bib3} and $\alpha$-cyclicity \cite{C_Beeri:bib4}. We show the relation between the Berge and matroidal cyclicities here:

\begin{defn}
A \textbf{Berge cycle} in a hypergraph $H$ is a sequence $$(v_1,e_1,v_2,e_2,\dots,e_{k-1},v_k,e_k,v_1)$$ such that
\begin{enumerate}
\item $e_1,\dots,e_k$ are distinct edges of $H$;
\item $v_1,\dots,v_k$ are distinct vertices of $H$;
\item $v_i,v_{i+1}\in e_i$ for $i=1,\dots , {k-1}$;
\item $v_k,v_1\in e_k$.
\end{enumerate}
\end{defn}



In order to focus on the dependency aspect of the edge set of a graph, in this paper we take a rather broad usage of the word "cycle". In simple terms, we will say that a set of edges is a matroidal cycle, if the edge set doubly covers its union, and is minimal with this property. More formally, we have the following definition.

\begin{defn}
A collection of edges $\{e_1,...,e_k\}$ in a hypergraph $H$ is \textbf{doubly covering} if $$e_i\subseteq\bigcup_{\stackrel{1\leq j\leq k}{j\neq i}} e_j$$ for all $i=1,\dots, k$. A doubly covering set is a \textbf{matroidal cycle} if it does not have any doubly covering proper subset.
\end{defn}

It is easy to see that the matroidal cyclicity implies Berge cyclicity but not $\alpha$-cyclicity. The reverse implication does not hold, since Berge cyclicity is preserved if new vertices are added to any of the edges, whereas matroidal cyclicity is not.

\begin{proposition}
     Let $S$ be a matroidal cycle in a hypergraph $H$. Then there is a Berge cycle in $H$, containing a subset of the edges in $S$.
\end{proposition}

\begin{proof}
If at least one of the intersections of $S$ contains two or more edges, we have a Berge cycle. Suppose that is not the case. Since $S$ is a matroidal cycle, all of the vertices must be in the intersections of the edges. Each intersection is defined by two edges and a vertex ($e_{i-1},x_i,e_i$). Take the first intersection and denote it $e_0,x_1,e_1$. Since $H$ is simple, $e_1$ must contain another vertex besides $x_1$, say $x_2$. $e_1$ and $x_2$, together with another edge, say $e_2$ define a new intersection. If $e_2=e_0$, we are done and have a Berge cycle. If that is not the case, then we can continue the algorithm by adding a new intersection. Since every intersection uses at least two edges, we will always have to either add a new edge or repeat an edge that was added by some previous intersection. If the latter is the case, we close the cycle. Since we have a finite number of intersections, we will run out of new edges and the cycle will eventually close.
    
\end{proof}

 To avoid unnecessarily wordy sentences, here we will refer to matroidal cycles as simply cycles. 

\section{Closures}

 As briefly mentioned earlier, \textbf{we consider matroids as a special kind of simple hypergraphs}, via their circuit sets. While there are, in general, many possible matroids that contain a prescribed hypergraph $H$ as a subgraph, we will define a canonical one next, and call that the {\em matroidal closure} of $H$. This is not a closure (or indeed canonical) in any categorical (or romantic) sense. 
 The definition is an adaption of a closure construction for a specific hypergraph $\mathcal{A}_0$ from \cite{R_Freij-Hollanti:bib5}.

\begin{defn}
Let $\mathcal{E}$ be the edge set of a hypergraph $H=(V,\mathcal{E})$. Denote by
\[
\epsilon(\mathcal{E}):=\mathcal{E}\cup\left\{(A_1 \cup A_2) \setminus \{v\} : A_1, A_2\in \mathcal{E}, A_1\cap A_2\not\in\mathcal{E}, v\in A_1\cap A_2\right\},
\]
\[
\min \mathcal{E}:=\{A\in\mathcal{E}: \not\exists A'\in\mathcal{E}: A'\subsetneq A\},
\]
and
\[
\uparrow \mathcal{E}:=\{A\in\mathcal{E}: \exists A'\in\mathcal{E}: A'\subseteq A\}.
\]
\end{defn}
 
 By construction, $(V,\min\mathcal{E})$ is a simple hypergraph for every hypergraph $(V,E)$. We will next show that recursive application of the composed operation $\min\epsilon$ yields a matroid $\overline{H}$, from any hypergraph $H$.

\begin{theorem}
Let $H=(V,E)$ be a hypergraph. Let $\mathcal{E}_0=\mathcal{E}$, $\mathcal{E}_{i+1}=\min\epsilon \mathcal{E}_i$ for $i\in\mathbb{N}$ and $\mathcal{E}'=\min\big{(}\cup_i E_i\big{)}$. Moreover, let $\mathcal{F}_0=\mathcal{E}$, $\mathcal{F}_{i+1}=\uparrow\epsilon \mathcal{F}_i$ for $i\in\mathbb{N}$ and $\mathcal{F}=\big{(}\cup_i F_i\big{)}$.
Then $\min\mathcal{F}=\mathcal{E}'$, and $\overline{H}=(V,\mathcal{E}')$ is a matroid. We call this the \textbf{matroidal closure} of $H$.
\end{theorem}
\begin{proof}
    We first observe that $\{\mathcal{F}_i\}_i$ is an increasing sequence of subsets of a finite set, so it terminates at $\mathcal{F}$, meaning that $\uparrow\epsilon\mathcal{F}=\mathcal{F}$. Therefore $\min\epsilon\min\mathcal{F}=\min\mathcal{F}$, so $\min\mathcal{F}$ satisfies the matroid axioms in Definition~\ref{MatroidDef}.

    By induction on $i$, we see that $\min \mathcal{F}_i=\mathcal{E}_i$ for every $i$. Indeed, every set of the form $A\cup B\setminus\{v\}$ for some $v\in A\cap B$, $A,B\in\mathcal{F}_i$, contains the set $A'\cup B'\setminus\{v\}$, where $A'\subseteq A$ and $B'\subseteq B$, and $A', B'\in\min\mathcal{F}$. Moreover, if $A\cap B\not\in \mathcal{F}_i$, then $A'\cap B'\not\in \min\mathcal{F}_i$. It follows that $$\mathcal{F}_{i+1}=\min\epsilon\mathcal{F}_i=\min\epsilon\min\mathcal{F}_i=\min \epsilon\mathcal{E}_i=\mathcal{E}_{i+1}$$
\end{proof}

The idea of the concept above is the same as the one of the combinatorial derived matroid. We will call the hypergraphs that are such closures \textit{matroidal hypergrapghs}. Now that we have a matroid, we can talk about what translation can be given for a notion of a rank.

\begin{defn}
The \textbf{rank of a matroidal hypergraph} is the largest set of vertices such that there are no edges containing only those vertices. We will call the vertices in said set a \textbf{basis}.
\end{defn}

\begin{defn}
Define the \textbf{rank of a hypergraph} as the rank of its matroidal closure.
\end{defn}

To make the proof sound less tangled, we will call the vertices from the basis, once we fix it, \textbf{independent} and all other vertices \textbf{dependent}.

\section{Hypergraphical matroids}

The idea of constructing matroids from hypergraphs is not new. For example, regular hypergraphs were analysed in \cite{R_Main:bib9} (Main, Roger Anthony (1978)). However, we did not find any literature generalising the definitions for non-regular hypergraphical matroids. We present our interpretation of the generalisation here, together with some resultant facts. We exploit the definition of the matroidal closure of a hypergraph, from the last section. 

To every graph $G=(V,E)$, it is a classical construction to associate a matroid $(E,\mathcal{C})$, whose ground set is $E$, where $\mathcal{C}$ consists of the (edge sets of) cycles in the graph. In \cite{M_Lorea:bib10}, this was generalized to arbitrary hypergraphs, yielding a very weak matroid.

\begin{defn}
A matroid $M(H)$, Lorea-associated with a hypergraph $H(V,\mathcal{E})$ is a pair $(\mathcal{E},I)$, where $\mathcal{I}\in I$ iff $\mathcal{I}$ is a collection of edges in $H(\mathcal{I}\subset\mathcal{E})$ 
\begin{itemize}
\item $\exists$ a simple and acyclic graph $G(V,U)$
\item $\exists$ a bijection $b$ from $\mathcal{U}$ to $\mathcal{I}$ such that $b(U)\supset U$ for all $U\in\mathcal{U}$
\end{itemize}
\end{defn}

\begin{defn}
The hypergraphs represented by the elements of $I$ are Lorea-cycle-free hypergraphs.
\end{defn}

This definition is then followed by a proposition, that can be easily compared to our results:

\begin{proposition}
A hypergraph has a Lorea's cycle iff
\[
|E|>|V|-1
\]
\end{proposition}

As we will see, Lorea's cycles are also cycles by our definition.

In \cite{R_Main:bib9}, a stronger generalization of graphical matroids was constructed for $k$-regular hypergraphs, by defining the independent sets of edges:

\begin{defn}
An edge set $E'\subset E$ of a $k$-hypergraph $H=(V,E)$ is called independent if either it is empty or $|V(G)|\geq |G|+k-1$ for each nonempty $G\subset E$. Every matroid that is isomorphic to $(E, I)$ for such a $k$-hypergraph is called the \textbf{$k$-hypergraphical matroid}.
\end{defn}

In order to generalize this definition to arbitrary hypergraphs, we need to settle on what edge sets will be considered cyclic.

In \cite{R_Freij-Hollanti:bib5} (Freij-Hollanti, Jurrius, Kuznetsova, 2023), the authors introduced the concept of a combinatorial derived matroid, liberating the secondary dependencies from the fixation of the representation field. Here we continue the theory, basing our research on the said interpretation of the derivation. 

In short, the construction of the derived matroid consists of taking the circuits of a given matroid as elements and using the matroidal rules to form a ``minimal'' matroid, closing it upwards. Here we briefly present the said operations and the necessary definitions.

\begin{defn}
Let $M=(E,\mathcal{D})$ be a matroid on the finite ground set $E$. We call the function $r:2^E\to \mathbb{N}$ \textbf{a rank function} if it is defined by 
\[
r(S)=\max\{|T|: T\subseteq S, T\not\in\mathcal{D}\}.
\]
The \textbf{nullity function} $\eta:2^E\to \mathbb{N}$ is then defined by $\eta(S)=|S|-r(S)$.
\end{defn}

\begin{defn}
Take a matroid $M$ with a collection of circuits $\mathcal{C}=\mathcal{C}(M)$. Denote 
\[
\mathcal{A}_0:=\{A \subseteq \mathcal{C}: |A|> \eta(\cup_{C\in A} C)\}.  
\]
Inductively, define $\mathcal{A}_{i+1}={\uparrow}\epsilon(\mathcal{A}_i)$ for $i\geq 1$, and 
\[
\mathcal{A}=\bigcup_{i\geq 0} \mathcal{A}_i.
\]
\end{defn}

\begin{defn}
    Consider a matroid $M$ with a collection of circuits $\mathcal{C}$. Then the \textbf{combinatorial derived matroid} $\delta M$ is a matroid which has a ground set $\mathcal{C}(M)$ with dependent sets being elements of $\mathcal{A}$.
\end{defn}

The next proposition states that the derived matroid construction decomposes into connected components. In particular, it will allow us to restrict attention to connected hypergraphs in several later instances.
\begin{proposition}\label{connected}
    Let $H=H_1\sqcup H_2$ be a disconnected hypergraph. Then $\overline{H}=\overline{H_1}\sqcup \overline{H_2}$ and $\delta H =\delta H_1 \sqcup \delta H_2$.
\end{proposition}
\begin{proof}
    The proof of the proposition follows from the definitions of the closure and the derived matroid. In both operations we only consider the edges that intersect with other edges, so when we deal with disconnected hypergraphs, we can apply the operations to each component separately. 
\end{proof}

When talking about hypergraphs that represent elements of matroids, the notation $\delta H$ will mean a hypergraph that represents a combinatorial derived matroid of a matroid that was represented by $H$. In the attempt to make things more clear, we suggest to think about the edges on $H$ as the vertex set of $\delta H$ and the cycles of $H$ as the edges of $\delta H$.

Since a lot of the reasoning here deals with an edge union representing the covered vertices, we simplify our notation:

\begin{notation}
Denote $\cup A:= \cup_{a\in A}a$.
\end{notation}

To tie the derived matroid theory to the visualisation of hypergraph cycles, we prove the fact below.

\begin{proposition}\label{prop-a}
Every element of $\mathcal{A}_0$ represents a set of edges that contains a cycle of a hypergraph.
\end{proposition}

\begin{proof}
Assume $A\in\mathcal{A}_0$ and let $T\subset A$ be inclusion minimal in $\mathcal{A}_0$. Assume $A$ has no cycles, and therefore $T$ is not a cycle.

Now let $C\in T$, $x\in C$ and $x\notin\cup T\setminus C$. Denote $S:=T\setminus \{C\}$.
Then 
\[
\eta(\cup T)+\eta(\cup S\cap C)\geq\eta(\cup S)+\eta(C)
\]
\[
\eta(\cup T)+0\geq\eta(\cup S)+1
\]
\[
|T|>\eta(\cup S)+1
\]
\[
|S|=|T|-1>\eta(\cup S)+1-1=\eta(\cup S),
\]
which contradicts the fact that $|T|$ was minimal. 

\end{proof}

Moreover, we can show that any element of $\mathcal{A}$ must be or contain a cycle. Now this gives us the power to jump between the visual concepts of cycles and the theory of derived matroids. As we will later see, this allows us to consider both sides in terms of another.

\section{Matroidal trees}

\begin{proposition}
The hypergraph $H$ contains a cycle if
\[
|E|>|V|-\rk(H).
\]
\end{proposition}

\begin{proof}
An edge set, satisfying the inequality, also satisfies $\{E\subseteq\mathcal{C}:|E|>n(\cup_{\mathcal{C}\in E}C)\}$ for some matroid $M$ with a collection of circuits $\mathcal{C}$. Therefore, it belongs to $\mathcal{A}_0$ of some $M$ and, by Proposition \ref{prop-a}, contains a cycle.
\end{proof}

\begin{defn}
We consider an edge to be \textbf{proper}, if it is neither a superset nor equal to some edge in the closure. 
\end{defn}

We want to think about proper edges as edges, adding new information about dependencies. Also, since we consider simple hypergraphs, it also cannot be a subset of an existing edge.

\begin{theorem}\label{a}
Adding a new proper edge to $H$ always decreases the rank.
\end{theorem}

\begin{proof}
Case 1: the new edge is a subedge of some edge from of the matroidal closure of~$H$.

Denote this edge $e^*$. 

Every new edge we make while constructing the closure, is a result of the union of two edges minus a vertex from their intersection. Let us call the said two edges \textit{parents} and the new edge their \textit{child}. Denote the parents of $e$ $A$ and $B$. Also denote the vertex we remove~$v$.

First, let us assume that both parents belong to the closure of $H'$.

\begin{lemma}\label{b}

No matter which vertex of $e$ we pick, there exists a subedge of $e$, smaller than $e$ and containing that vertex.
\end{lemma}

\begin{proof}
We construct new edges $A'$ and $B'$ by removing some element from intersections with $e^*$ like so:
\[
A'=A\cup e^*\setminus v'
\]
\[
B'=B\cup e^*\setminus v''
\]

Here we choose $v'$ and $v''$ so that at least one of them does not belong to the intersection of $A$ and $B$.

What we know:
\begin{itemize}
    \item $v'$ and $v''$ $\in e$
    \item $v\notin e$
    \item $v\in A$, $v\in A'$
    \item $v\in B$, $v\in B'$
\end{itemize}

What is left to do, is to take an intersection of our newly created edges and remove~$v$:
\[
C=A'\cup B'\setminus v
\]

Since at least one of the vertices $v'$ and $v''$ does not belong to the intersection of $A$ and $B$, the child $C$ will not be equal to $e^*$ and therefore will be a subset of it. We were free to choose any vertex that does not belong to $A\cap B$, which means that we are able to generate any subset of $e$, not containing such vertex.

If we aim to generate a subedge of $e$ without one of the vertices from $A\cap B$, all we need to do is to remove the same vertex from both, $A$ and $B$, so that the parents $A'$ and $B'$ do not contain it. Then the child $C$ will not contain it as well.

The Lemma \ref{b} is proved.
\end{proof}

\begin{lemma}\label{c}
The rank of the subhypergraph made by $A$, $B$ and $e$ is greater than the rank of the graph, induced by $A$, $B$ and $e^*$.
\end{lemma}

\begin{proof}
It is easy to see that the rank of $A\cup B\cup e$ is $|e|-1$. Since by Lemma \ref{b}, any of the vertices of $e^*$ is in some subedge of the graph, induced by $A$, $B$ and $e^*$, the rank must be at least 1 vertex smaller.
\end{proof}

Now consider the case where at least one of the parents is not a part of the matroidal closure of $H'$. Wlog assume that it is the edge $A$ and before constructing it we were able to construct $A'\subsetneq A$. Then we treat $A'$ the same way we treated $e^*$. We consider the parents of $A$ and using Lemma \ref{b} prove that the vertices of $A$ will all now belong to some subsets of $A$. If the parents of $A$ are also not in $H'$ because of their subsets being made first, we consider them as new vertices and again apply Lemma \ref{b} and so on. What we get in the end is that all edges that were affected by their subsets being constructed first, will now be "cut" into smaller subsets and therefore all vertices of edge $e$ will belong to some subsets of $e$. By Lemma \ref{c}, the rank of the induced subhypergraph on $V(A\cup B)$ will have a smaller rank on $H'$ than on $H$.

\begin{lemma}
The local rank change, caused by the new edge $e^*$, changes the rank of the hypergraph.
\end{lemma}

\begin{proof}
Suppose that is not the case. We consider the closures $\overline{H}$ and $\overline{H'}$.

Note that every independent vertex must belong to at least one edge that has only one dependent vertex. More so, every dependent vertex belongs to some edge, where it is the only dependent vertex.

Since all of the vertices of $e$ are in some smaller subsets of $e$, there must be at least two vertices of $e$ that are dependent.

Since this is the closure, every dependent vertex belongs to an edge such that all but one vertices contained in it are independent. Let's call the set of such edges \textit{special}.

Now we will describe an algorithm that finds an edge, containing only dependent vertices.

Algorithm: we take $e$ and all of the special edges that share vertices with $e$. We apply the union minus a vertex operation to $e$ and one of the edges that share a dependent vertex. The new edge $e^{(1)}$ now contains more vertices from independent set and one less vertex not from independent set. If we run into a subedge $S*$ of $e^{(i)}\cup S$, for some special edge $S$, we apply the operation on $e^{(i)}$ and $S*$ first and remove some other dependent vertex from their union. The subedge must contain other dependent vertices of $e$ (otherwise it would only contain independent vertices). After this, $e^{(i+1)}\cup S$ no longer contains $S*$ and does not prevent us from the union minus a vertex operation on $e^{(i+1)}$ and $S$. We repeat the process until our new edge $e^{(k)}$ only contains the vertices from the dependent set. This contradicts our statement. $\lightning$
\end{proof}

Case 2: the new edge $e^*$ is not a subedge of the matroidal closure of $H$.

Lemma \ref{b} shows that any edge, that cannot be constructed because of its subedge being constructed first, will be "cut" into smaller subedges.

We look at the closure of $H'$. Because of the possible "cuts" that were made on the edges of $H$, the maximal independent sets of $H$ and $H'$ will have to use the same vertices (this is under assumption that the rank did not decrease). Because of this, we can take some basis of $H$ and analyse it on $H'$.

Assume there exists a basis on $H$ such that at least one vertex of $e^*$ is dependent. Since we had the same basis in $H$, every dependent vertex is in some special edge that is not equal to $e^*$. This means that we can apply the same algorithm on $e^*$ and produce an edge containing only dependent vertices. $\lightning$

This proves Theorem \ref{a}.
\end{proof}

\begin{theorem}\label{ineq}
For all hypergraphs,
    \[
    |E|\geq|V|-\rk(H).
    \]

\end{theorem}

\begin{proof}
Suppose that there exists a hypergraph $H_1$ with $|E|<|V|-\rk(H_1)$. We can add new edges to it until it is maximal while still satisfying the inequality. We call this hypergraph~$H$.

We now add a new \textbf{proper} edge $e^*$ to $H$ and call our new hypergraph $H'$.

Note: we can always find a proper edge to add. If that was not the case, we would be able to keep adding edges until our hypergraph turns into complete graph and for any complete graph $|E(G)|=\binom{n}{2}>n-1=|V(G)|-\rk(G)$.

Since $H$ was maximal, we know that such inequality must hold:

\[
|E(H')|\geq|V(H')|-\rk(H').
\]

Considering the relations between $H$ and $H'$, we get such system of inequalities:
\[
\begin{cases}
|E(H)|<|V(H)|-\rk(H)\\
|E(H)|+1\geq |V(H)|-\rk(H').
\end{cases}
\]

It follows that 
\[
\rk(H)-\rk(H')\leq1.
\]

Since the rank function can only have natural values, this means that adding $e^*$ should not change the rank.


By the Theorem \ref{a}, every proper edge decreases the rank of $H$, so we get a contradiction.~$\lightning$


\end{proof}

\begin{defn}
A hypergraph with the minimal amount of edges required to describe its dependencies, will be called a \textbf{matroidal tree}.
\end{defn}

\begin{defn}
If a hypergraph $H$ with $|E|=|V|-\rk(H)$ does not contain a cycle, we will call it a \textbf{natural tree}.
\end{defn}

Note that, for example, regular graph trees are natural trees, since the equation is equivalent to $|E|=|V|-c(G)$, where $c(G)$ denotes the number of connected components.

In case the distinction between matroidal and natural trees would be brought into question, we offer an example of a hypergraph that is a matroidal tree but not a natural tree:

\begin{center}
\includegraphics[scale=0.5]{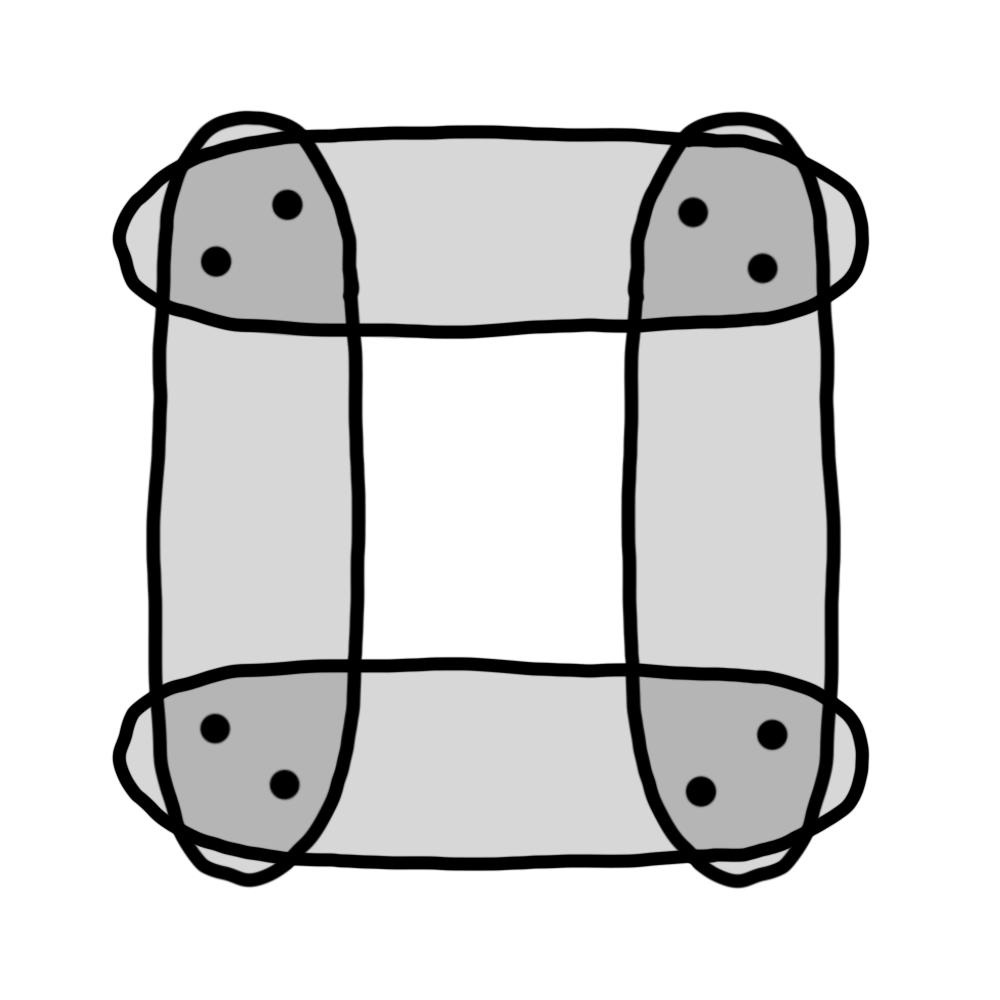}
\end{center}

As one can see, a double covering does not ensure the "matroidal dependency". In Section 6, we conjecture that such hypergraphs, that are not natural trees, cannot be derivatives of matroids. This example also suggests that such hypergraphs might describe a different type of dependency, that is not "algebraic".

\section{Closure families}

Definition of a closure describes a structure that is shared amongst a number of hypergraphs. In this section, we categorise those hypergraphs according to what closure they produce and analyse the properties of these categories. The idea of how one of the said categories will look like is formalised here:

\begin{defn}
For a matroid $M$, let $\mathcal{H}_M=\{H:\overline{H}=M\}$. A family of the form $\mathcal{H}_M$ for some $M$ is called a \textbf{closure family}.
\end{defn}

\begin{defn}
A hypergraph that belongs to a family that contains a natural tree, will be called \textbf{natural}.  
\end{defn}

\begin{defn}
We call a family $\delta\mathcal{H}$ a \textit{derived family} of $\mathcal{H}$ if for $H\in\mathcal{H}$, $\delta\overline{H}\in\delta\mathcal{H}$.
\end{defn}

Now in order to be able to analyse the behaviour of the families of higher derivations, we define the $k$'th combinatorial derivative of a matroid.

\begin{defn}
We define the $k$-th derivative as $\delta_{k+1}H:=\delta\delta_k H$. If $\delta_{k}H=\delta_{k+1}H$ for some $k\in\mathbb{N}$, then we will say that the derived hypergraph \textbf{converges} and has a \textbf{limit} $\lim_{\delta}H=\delta_{k}H$. Otherwise, we say that the hypergraph \textbf{diverges}.   
\end{defn}

\begin{defn}
We say that a closure family $\mathcal{H}$ \textbf{fades} if $\lim_\delta\mathcal{H}=\emptyset$.
\end{defn}

\begin{ex}
One example of a fading family could be a closure family $\mathcal{H}_1$, depicted (up to isomorphism) in the image below.
\begin{center}
    \includegraphics[scale=0.5]{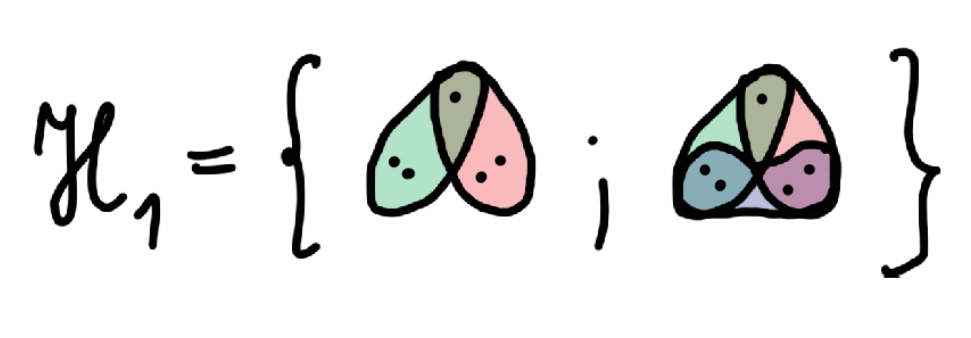}
\end{center}

The derived family of $\mathcal{H}_1$ contains only one hypergraph that has a single edge.

\begin{center}
    \includegraphics[scale=0.5]{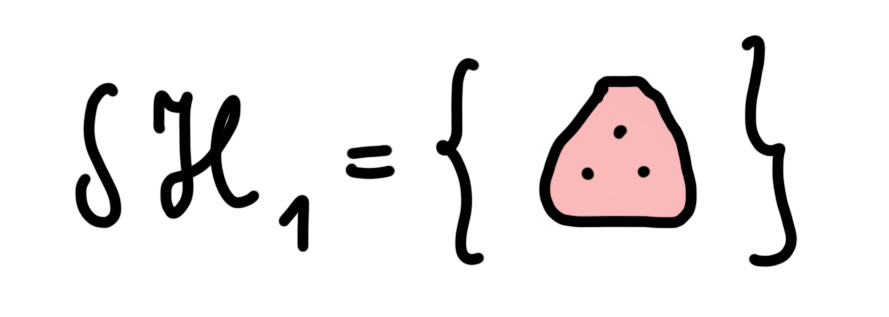}
\end{center}

The derived family of $\delta\mathcal{H}_1$ is empty since there were no cycles in $\delta\mathcal{H}_1$. Therefore, we can conclude that $\mathcal{H}_1$, as well as $\delta\mathcal{H}_1$, fades.
\end{ex}

\begin{ex}\label{ex}
    
A good example of a not fading family is the closure family $\mathcal{H}_2$, which contains a complete $3$-regular hypergraph $K_4^3$.

\begin{center}
\includegraphics[scale=0.5]{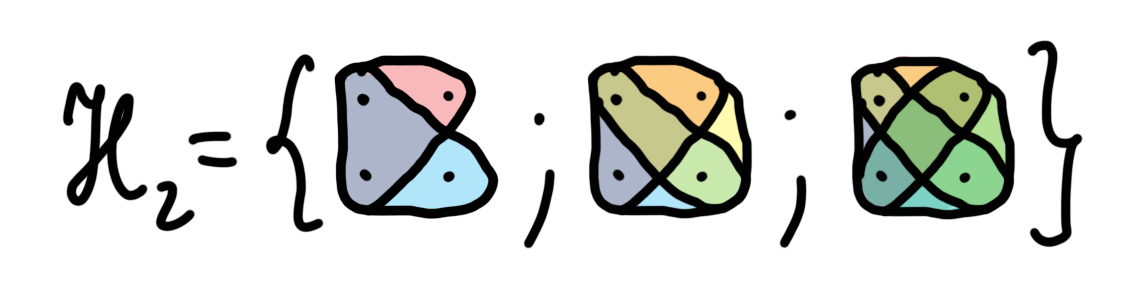}
\end{center}

The family does not fade, since it is equal to its closure:

\begin{center}
\includegraphics[scale=0.5]{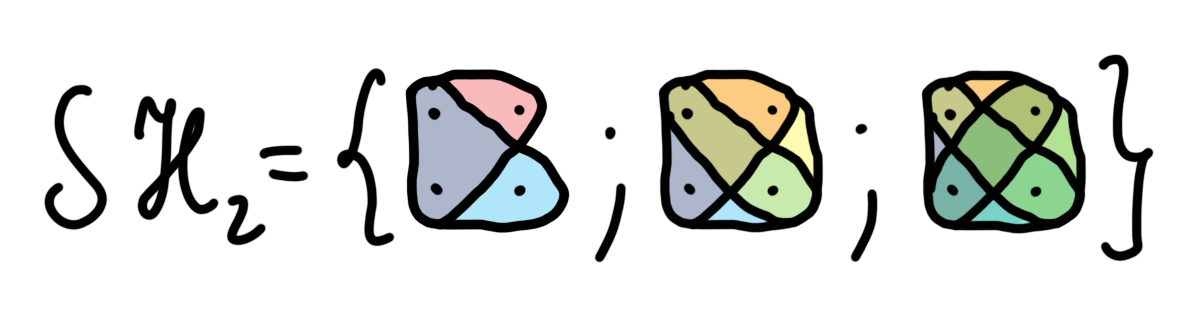}
\end{center}

\end{ex}

\begin{proposition}\label{NoFade}
Every closure family, that contains a hypergraph with $K^3_4$ as a subgraph, does not fade.
\end{proposition}
\begin{proof}
Since $\delta K_4^3=K_4^3$, the limit of such hypergraph is $K_4^3\neq\emptyset$.
\end{proof}

\begin{corollary}\label{regular}
Every closure family, that contains a hypergraph with an $r$-regular clique, where $r>2$, as a subgraph, does not fade.
\end{corollary}







\begin{conjecture}
A family containing a tree that is not natural, cannot be the derivative of a matroidal family.
\end{conjecture}

\section{Iterated derived matroids}

In this section, we consider the question of what matroids can occur as limits of hypergraph families. In other words, we seek to find all matroids $M$ such that $\delta M = M$. This question was considered for the related construction of {\em represented} derived matroids by Oxley and Wang~\cite{OW:bib12}, and our analysis rather closely follows theirs, although the construction of $\delta M$ is quite different (and in particular a matroid invariant) in our case. We will use the following two lemmas.

\begin{lemma}[Part of Lemma 14 in~\cite{OW:bib12}]\label{LmNofCirc}
    Let $M$ be a nonempty connected matroid. Then $$|\mathcal{C}(M)|\geq\binom{\eta(M)+1}{2}.$$ 
\end{lemma}

\begin{lemma}[Lemma 17.5 in~\cite{R_Freij-Hollanti:bib5}]\label{LmEtaBd}
    Let $\delta M$ be the derived matroid of a connected matroid~$M$. Then $$\rk(\delta M)\leq\eta(M).$$ 
\end{lemma}

It was shown by Knutsen that the inequality in Lemma~\ref{LmEtaBd} can sometimes be strict, and in particular that this is the case for the Vamos matroid \cite{K_Teodor:bib13}. However, it is conjectured in \cite{R_Freij-Hollanti:bib5} that equality holds whenever $M$ has a coadjoint.

By Proposition~\ref{connected}, it suffices to study the derived matroids of connected hypergraphs. To simplify statements in what follows, we will use the following definition.

\begin{defn}
    A matroid $(E,\mathcal{C})$ is called a {\em tricycle} if $|\mathcal{C}|=3$ with $\mathcal{C}=\{C_1, C_2, C_3\}$, such that $C_i\cap C_j\neq\emptyset$ and $C_i\cup C_j=E$ for all $i,j$, and $\cap_{i=1}^3 C_i = \emptyset$.
\end{defn}

It is straightforward to see that a tricycle is precisely the graphical matroid of a theta graph, which is the union of three independent paths $C_1\cap C_2$, $C_1\cap C_3$ and $C_2\cap C_3$ between the same pairs of vertices. Moreover, the rank of a tricyle $(E,\mathcal{C})$ is $|E|-2$, a basis being obtained by removing one element from each of $C_1\cap C_2$ and $C_1\cap C_3$. The results in~\cite{OW:bib12} are phrased in terms of theta graphs, but are fully analogous to ours.

\begin{lemma}\label{LmFade}
Let $M$ be a connected matroid. Then
\begin{enumerate}[i)]
\item $\delta M=\emptyset$ if and only if $M=U_1^1$.
\item $\delta M = U_1^1$ if and only if $M=U_{n+1}^n$ for some $n$.
\item $\delta M= U_{n+1}^n$ for some $n$ if and only if $M$ is a tricycle and $n=2$.
\end{enumerate}
\end{lemma}

\begin{proof}
    \begin{enumerate}[i)]
    \item A connected matroid with more than one element always has a circuit, so the only connected matroid with $\delta M=\emptyset$ is the one element simple matroid $U_1^1$.
    \item A connected matroid with a single circuit has to have every element contained in that circuit, so be isomorpic to $U_{n+1}^n$. Conversely, $\delta(U_{n+1}^n)$ is a simple matroid with a single element. Therefore it cannot have any circuit, so it equals $U_1^1$.
    \item Clearly, if $M$ is a tricycle, then $\delta(M)$ is a three element matroid with a single circuit consisting of all three elements, i.e. $\delta M=U_3^2$. Conversely, if  $\delta M=U_{n+1}^n$, we have \begin{align*}1&=\eta(\delta M)=|\mathcal{C}(M)|-\rk(\delta(M))\\&\geq |\mathcal{C}(M)|-\eta(M)\\&\geq \binom{\eta(M)+1}{2}-\eta(M)=\binom{\eta(M)}{2},\end{align*}
    where the first inequality follows from Lemma~\ref{LmEtaBd} and the second one from Lemma~\ref{LmNofCirc}, so $\eta(M)=2$ and $|\mathcal{C}(M)|=3$. The only simple matroids with three circuits and nullity two are tricycles.
    \end{enumerate}
\end{proof}

\begin{lemma}\label{LmFade2}
A tricycle is not the derived matroid of any matroid.
\end{lemma}

\begin{proof}
    If $\delta(M)$ is a tricycle, then \begin{align*}2&=\eta(\delta M)=|\mathcal{C}(M)|-\rk(\delta M)\\&\geq |\mathcal{C}(M)|-\eta(M)\\&\geq \binom{\eta(M)+1}{2}-\eta(M)=\binom{\eta(M)}{2},\end{align*} where the first inequality follows from Lemma~\ref{LmEtaBd} and the second one from Lemma~\ref{LmNofCirc}. It would then follow that $\eta(M)\leq 2$, and so $\rk(\delta M)\leq\eta(M)\leq 2$. Since $\delta M$ has more than one element, $M$ has more than one circuit, and so can not have nullity $<2$.

    We only need to study the case where $\eta(M)=2$ and $|\mathcal{C}(M)|\in\{3,4\}$. But then at most one of the three intersections $C_i\cap C_j$ of circuits in the tricycle has more than one element, and so the tricycle has some circuit of size $2$. This contradicts the fact that a derived matroid is always simple. 
\end{proof}

\begin{theorem}\label{Fixed}
Let $M$ be a matroid with $\delta M \cong M$. Then $M$ is a direct sum of matroids, each of which is isomorphic to $U^2_4$.
\end{theorem}
Note that the uniform matroid $U^2_4$, considered as a hypergraph, is the complete $3$-regular hypergraph $K_4^3$. This is thus the example preceding Proposition~\ref{NoFade}.
\begin{proof}
    $U_4^2$ has four circuits, where each triple doubly covers the ground set, and so forms a circuit in $\delta U_4^2$. Thus $\delta(U_4^2)\cong U_4^2$. 

    Now assume $M$ is connected with $\delta M\cong M$. Lemmas \ref{LmEtaBd} and \ref{LmNofCirc} show that $\eta(M)=\eta(\delta M)\geq\binom{\eta(M)}{2}$, which can only happen if $\eta(M)\leq 3$. Lemma \ref{LmFade}, parts (i) and (ii) show that $\delta(M)\not\cong M$ for all connected matroids with nullity $0$ or $1$, so we only need to consider the case $\eta(M)\in\{2,3\}$. 
    
    If $\eta(M)=3=\binom{\eta(M)}{2}$,  the inequalities in Lemma \ref{LmNofCirc} and \ref{LmEtaBd} must both be met with equality, so $|\mathcal{C}(M)|=\binom{3+1}{2}=6$ and $\rk M=\rk(\delta M)=\eta(M)=3$. But it is readily seen by exhaustive search that no simple connected matroid with rank $3$, size $6$ and $<7$ circuits exist. (See also \cite{OW:bib12}, Lemma 22.)

    Remains the case $\eta(M)=2$. By Lemma \ref{LmEtaBd}, we have $$\rk(M)=\rk(\delta M)\leq \eta(M)=2,$$ and so $M$ is a simple matroid with rank $2$, i.e. $M\cong U_{2+\eta(M)}^2=U_4^2$.

    Finally, if $\delta(M)\cong M$, then the same has to hold for each connected component of $M$, since the derived matroid preserves direct sums. Therefore, each component of $M$ must be isomorphic to $U_4^2$. This completes the proof.
    \end{proof}

The following theorem closes the description of hypergraph limits, as it shows that all families diverge, apart from the the ones considered in Lemmas \ref{LmFade} and \ref{LmFade2} and Theorem \ref{Fixed}.

\begin{theorem}
Let $M$ be a connected matroid on at least two elements. If $M$ is a tricycle or $M=U_{n+1}^n$ is a single circuit, then $\mathcal{H}_M$ fades. If the cosimplification of $M$ is $U_4^2$, then $\mathcal{H}_M$ converges to $U_4^2$ (or in hypergraph terms, $K_4^3$). Otherwise, $\mathcal{H}_M$ diverges, and $\eta(\delta^{k+1}M)>\eta(\delta^k M)$ for all $k\geq 0$.
\end{theorem}

\begin{proof}
According to Lemma~\ref{LmFade}, $\delta^2 U_{n+1}^n=\emptyset$, and if $M$ is a tricycle then $\delta^3 M=\emptyset$, so in these cases, $\mathcal{H}_M$ fades.

If the cosimplification of $M$ is $U_4^2$, then the hypergraph of $M$ has precisely four edges, and each triple of edges forms a cycle, so $\delta(M)$ is represented by the hypergraph $K_4^3$ and, as shown in the example \ref{ex}, its family converges to itself.


Now, let $M$ be an arbitrary simple connected matroid. Lemmas \ref{LmEtaBd} and \ref{LmNofCirc} show that, if $\eta(M)\geq 4$, then $\eta(\delta M)\geq\binom{\eta(M)}{2}>\eta(M)\geq 4$, so the sequence $\eta(\delta^k M)$ is strictly increasing in $k$ and the sequence diverges. The same obviously holds if $\eta(\delta M)\geq 4$. If $\eta(M)=1$, we have the previously studied case $M=U_{n+1}^n$, and if $\eta(\delta M)=1$ we have the previously studied case where $M$ is a tricycle. Hence, we only have the cases $\eta(M),\eta(\delta M)\in\{2,3\}$ remaining.

If $\eta(M)=2$ and $\eta(\delta M)=2$, then $$|\mathcal{C}(M)|=\rk(\delta M)+\eta(\delta M)\leq \eta(M)+\eta(\delta M)=4.$$ So $M^*$ is a matroid of rank 2 with precisely $4$ non-trivial flats, so the simplification of $M^*$ is $U_4^2$, i.e., the cosimplification of $M$ is $(U_4^2)^*\cong U_4^2$.

If $\eta(M)=3$ then $\eta(\delta M)\geq\binom{\eta(M)}{2}=3$, so the combination $\eta(M)=3$ and $\eta(\delta M)=2$ is not possible. If $\eta(M)=\eta(\delta M)=3$, then the equalities in \ref{LmEtaBd} and \ref{LmNofCirc} both have to be satisfied with equality, so $M$ must have exactly $\binom{\eta(M)+1}{2}=6$ circuits and $\delta M$ must have rank $\rk(\delta M)=\eta(M)=3$, and size $\rk(\delta M)+\eta(\delta M)=6$. Again, like in the proof of Theorem~\ref{Fixed}, we see that no such simple connected matroid exists. Therefore, $\eta(\delta M)>\eta(M)$ for all connected matroids $M$ of nullity $\geq 3$, as well as for connected matroids with $\eta(M)=2$ and $\eta(\delta(M))=3$. For all such matroids we therefore have $\eta(\delta^{k+1}M)>\eta(\delta^k M)$ for all $k\geq 0$, and $\mathcal{H}_M$ diverges.

\color{black}




\end{proof}

\end{document}